\documentclass[12pt]{amsart}

\usepackage{amsfonts,amssymb,amscd,amsmath}
\usepackage[colorlinks]{hyperref}

\usepackage{mathtools} % for xmapsto

\usepackage{graphicx, color}

\input{macro.sty}

\begin{document}

\dedicatory{Dedicated to the memory of Todd A. Drumm}

\title[Mapping class group action]
{The mapping class group action on $\mathsf{SU}(3)$-character varieties}

\author[Goldman]{William M. Goldman}
\address{
Department of Mathematics,
University of Maryland,
College Park, MD 20742, USA}
\email{wmg@umd.edu}

\author[Lawton]{Sean Lawton}

\address{Department of Mathematical Sciences, George Mason University,
4400 University Drive,
Fairfax, Virginia  22030, USA}

\email{slawton3@gmu.edu}

\author[Xia]{Eugene Z. Xia}
\address{
Department of Mathematics,
National Cheng Kung University,
Tainan 701, Taiwan}
\email{ezxia@ncku.edu.tw}

\date{\today}

%\thanks{}



\begin{abstract}
Let $\Sigma$ be a compact orientable surface of genus $g=1$ with $n=1$ boundary component.
The mapping class group $\Gamma$ of $\Sigma$ acts on the 
$\SU(3)$-character variety of $\Sigma$.
We show that the action is ergodic with respect to the natural symplectic measure on the character variety.
\end{abstract}

\maketitle

\tableofcontents

\section{Introduction}
% \marginpar{I expanded the introduction at the request of the referee.}

Let $\Sigma = \Sigma_{g,n}$ be the compact oriented surface of genus $g$
with boundary $\partial\Sigma$ which has $n\geq 1$ components
denoted $\partial_1\Sigma,...,\partial_n\Sigma$.
Fix a basepoint $x_0$ on $\Sigma$ and let $\pi = \pi_1(\Sigma,x_0)$ denote its fundamental group.
%\marginpar{I think $x_0$ must be on the boundary of $\Sigma$ because we don't want $\Gamma$ to move $\x_0$.}
Let $\HomeoS$ be the group of 
%The {\em mapping class group\/} $\Gamma$ consists of
% isotopy classes of 
orientation-preserving homeomorphisms of $\Sigma$
which fixes $\partial\Sigma$ pointwise.
Define the {\em mapping class group\/} of $\Sigma$:
\[ \Gamma : = \pi_0\big(\HomeoS\big).\]
Alternatively, 
choose basepoints $x_i\in\partial_i\Sigma$ and paths  from $x_0$ to $x_i$.
%\marginpar{This sentence is a bit confusing.  What is is being ``determined'' here by the paths? I changed it to ``between $x_0$ and $x_i$''.  Please check whether this is what is meant. }
Define $\pi_1(\partial_i\Sigma)$ as the cyclic subgroup of $\pi$ 
corresponding to $\pi_1\big(\partial_i\Sigma,x_i\big)$
and determined by the paths between $x_0$ and $x_i$.
Let $\AutS$ denote the subgroup of $\Aut(\pi)$ 
which preserves both the conjugacy classes of the 
subgroups $\pi_1\big(\partial_i(\Sigma)\big)$ 
and the orientation of $\Sigma$.
Then $\Gamma$ is  isomorphic to the image of $\AutS$ under the quotient
homomorphism
\begin{equation*}
\Aut(\pi) \longrightarrow \Out(\pi) := \Aut(\pi)/\Inn(\pi).
\end{equation*}

% all automorphisms of $\pi$ which preserve the conjugacy classes
%of the cyclic subgroups $\pi_1(\partial_i)\subset\pi$ and correspond
%to orientation-preserving homeomorphisms.

Let $G$ be an algebraic group over $\R$.
Then the set of homomorphisms $\pi \to G$ 
enjoys the the structure of an $\R$-algebraic set
denoted $\Hom(\pi,G)$.
Choose conjugacy classes $\CC_i \subset G$
for each $i=1,\dots,n$ and let $\CC := \{\CC_1,\dots,\CC_n\}$.
Denote by $\HomC$ the subset of $\Hom(\pi,G)$ comprising homomorphisms which send 
$\pi_1(\partial_i\Sigma)$ to $\CC_i$
 
The group $\Inn(G)$ of inner  automorphisms of $G$
acts on $\HomC$ by composition. 
Denote the resulting {\em relative character variety\/} by:
\[
\M_\CC(G) := \HomC/\Inn(G).
\]

The group $\AutS$ acts on $\pi$, 
and hence  on $\Hom_\CC(\pi,G)$ by composition.
Furthermore, the action descends to a $\Gamma$-action on $\MCC(G)$. 
The moduli space $\MCC(G)$ has an invariant dense open subset 
%$\MCC^o(G)$
$\MCCo$ which is a smooth manifold. 
If, for example, $G$ is $\R$-reductive (see \cite{Go84, Go97}), 
this subset has an $\Gamma$-invariant symplectic structure $\Omega$.
In particular, 
$\MCCo$  admits a natural smooth $\Gamma$-invariant measure $\mu$.

%Suppose that $\surf := \surf_{g,n}$ is an orientable surface of genus $g>0$ with boundary of $n$ disjoint circles.  Let $\MCC(K)$ be
%its $K$-character variety, with the symplectic structure $\Omega$ as defined in
%\cite{Go84, Go97}, where $\mathcal{C}$ is a collection of $n$ conjugation classes.
%Then $\Gamma$, the mapping class group of $\surf$, acts on $\MCC(K)$ preserving
%$\Omega$.
%Pickrell-Xia~\cite{PX02a, PX03a} established $\Gamma$-ergodicity of the symplectic measure $\mu$ of $\Omega$ for $g>1$ or $n > 2$ with $g > 0$.  This was previously proved by
%Goldman~\cite{Go97} when the simple factors of $K$ are locally isomorphic to
%$\SUtwo$.

This paper is part of the general program to understand the dynamics of the
action of $\Gamma$ and automorphism groups of free groups
on character varieties when the Lie group is compact.
Suppose that $K$ is a compact Lie group.
%\marginpar{Deleted a redundant sentence.}
%Let $K$ be a compact Lie group.  
In \cite{Go97}, Goldman conjectures that $\Gamma$ acts ergodically on 
$\MCC(K)$ and shows this to be the case when $K$ is locally a product of 
$\SU(2)$'s and $\U(1)$'s.  
In \cite{PX02a} and \cite{PX03a}, the conjecture is proved for $g \geq 2$, and also established for almost all boundary classes when $g=1=n$.  In \cite{GX11}, Goldman and Xia offer an alternate and simpler proof for the case of $K = \SU(2)$.

%In this note, we provide further evidence for the full conjecture of Goldman with the following theorem. 

In this paper, we consider the case $g = 1 = n$ with $K = \SUthree$.  
Then $\Gamma \cong \SL(2,\Z)$  (see \cite{FM12}). 
In this special case, $\MCC(K)$ is explicitly described by the 
{\em commutator map\/}
\begin{align*}
K \times K &\xrightarrow{~\kappa~}  K \\
(a,b) &\longmapsto [a,b] := aba\inv b\inv.
\end{align*}
%Indeed, the automorphism group
%$\Aut(\Ftwo)$ acts on the set $\Hom(\Ftwo,K) \approx K \times K$ .
%
%two-generator free group $\Ftwo$ 
%be the free group of two generators.
%
%???????????
%
%Then $\Aut(\Ftwo)$ acts on the set $\Hom(\Ftwo,K) \approx K \times K$ .
%The quotient of $\Out(\Ftwo) \cong \GLtwoZ$ acts on the quotient
%$\Hom(\Ftwo,K)/\Inn(K)$, where $\Inn(K)$ is the group of inner automorphisms of $K$.
%The index two subgroup $\Out^+(\Ftwo)$ corresponding to 
%$\SLtwoZ\subset \GLtwoZ$ 
%acts on 
Indeed, $\MCC(K) \cong \kappa\inv(\CC)/\Inn(K)$, 
for each conjugacy class $\CC \subset K$.
%

%\marginpar{The notation $\CC$ both denotes boundary and the conjugacy classes.}

\begin{thm}\label{thm:main}
Let $K = \SU(3)$ and $\surf=\surf_{1,1}$. 
The $\Gamma$-action is ergodic on $\MCC(K)$ with 
respect to the measure $\mu$.
\end{thm}
%\marginpar{The referee requested an outline of the proof, so I added this part in.}

We prove this theorem along the lines of the main results in \cite{GX11}.  If $\rho \in \HomC$, we denote its $\Inn(G)$-equivalence class as $[\rho]$.  Similarly if $S \subseteq \HomC$, then the corresponding set of $\Inn(G)$-equivalence classes is denoted by $[S]$.
A simple closed curve $\alpha$ on $\Sigma$
defines a function 
\begin{align*}
\MCC(K) &\xrightarrow{~\f_\alpha~} \C \\
[\rho] &\longmapsto \Tr\big(\rho(\alpha)\big).
\end{align*} 
The symplectic structure $\Omega$ together with the real and imaginary parts of $\f_\alpha$ give rise to Hamiltonian flows.  
The ring-theoretical results in \cite{La07} imply the algebra of Hamiltonian
vector fields is {\em infinitesimally\/} transitive.
It follows that the group generated by these flows is {\em locally\/} transitive, and hence, ergodic.

%$\f_\alpha$ on $\MCC(K)$ by $\f_\alpha([\rho]) = \Tr(\rho(\alpha)).$  The symplectic structure $\Omega$ together with the real and imaginary parts of $\f_\alpha$ give rise to Hamiltonian flows.  From the ring-theoretical results in \cite{La07}, we conclude that the group generated by these flows is locally transitive, hence, ergodic.
Depending on the choice of $\alpha$, these Hamiltonian flows preserve the sets $[\H(a, b)]$ or $[\H'(a,b)]$ defined in Section~\ref{subsec:DH}.
%These Hamiltonian flows preserve the sets $[\H(a,b)]$ (or $[\H'(a,b)]$ depending on the choice of $\alpha$.  See \S 5.1 for details).  
On the other hand, $\Gamma$ contains the Dehn twist $\tau_\alpha$ along $\alpha$.  The $\tau_\alpha$ action also preserves $[\H(a,b)]$ and is ergodic in almost all $[\H(a,b)]$. 
%\marginpar{I don't understand this sentence.  What is $\mu$-almost every level sets?  Equivalence relation generated by...  If we want to say this, then we need to say that the equivalence is only in the measure theoretical sense, but this is exactly what next sentence says which is precise, so I suggest that we delete this sentence.}
%On $\mu$-almost every level sets,
%the measurable equivalence relation generated by 
%$\langle\tau_\alpha\rangle$
%agrees with that defined by the Hamiltonian flows of $\f_\alpha$.  
It follows that a $\mu$-measurable function is
invariant under $\tau_\alpha$ if and only if it 
is invariant under the Hamiltonian flows associated with $\f_\alpha$.
By local transitivity, such a measurable function is almost everywhere constant.

%The variety $\M_c := \MCC(\SU(3))$ has a symplectic structure.  Each simple closed curve $\alpha$ on $\Sigma$ defines a The group $\Gamma$ is generated by Dehn twist along simple closed curves of $\Sigma$.  For each such curve $\alpha$, the Dehn twist $\tau_\alpha$ 
%showing that the traces of
%simple closed curves generate the coordinate ring of
% $\MCC$.  This latter ring-theoretic statement for $K = \SU(3)$ follows from results in \cite{La07}.

\subsection*{Acknowledgements}
Goldman and Lawton were partially supported by U.S. National Science Foundation grants: DMS 1107452, 1107263, 1107367 1309376 ``RNMS: Geometric structures And Representation varieties" (the GEAR Network).  Lawton acknowledges support from U.S. National Science Foundation grants: DMS 1309376, and also DMS 0932078000 while he was in residence at the Mathematical Sciences Research Institute in Berkeley, California during the Spring 2015 semester.  He was also partially supported by a Simons Foundation Collaboration grant. Xia was partially supported by the Ministry of Science and Technology Taiwan with grants 103-2115-M-006-007-MY2, 105-2115-M-006-006, 106-2115-M-006-008 and 107-2115-M-006-009 and the National Center for Theoretical Sciences, Taiwan.

\section*{Notation and terminology}

%Denote $\omega := e^{2\pi i/3}$ and the identity transformation by $\Id$.
% , $H \subgroupeq G$. and $A \in G$.   Denote its group of automorphisms by $\Aut(G)$.
%
Let $G$ be a group.  Denote the inner automorphism induced by $A\in G$ by:
\begin{align*}
G &\xrightarrow{~\Inn(A)~} G \\
B  &\longmapsto ABA\inv
\end{align*}
%\Ad: G  \lto G , \ \ \ \Ad(A) B =  ABA\inv,
%\end{align*}
%defining the homomorphism $\Ad : G \lto \Aut(G)$.
The set of conjugacy classes in $G$ equals the quotient $G/\Inn(G)$,
and we denote the image of a subset $S\subseteq G$ under the quotient map
by $[S]$.
 %\marginpar{b:I like $\Inn$ better than $\Ad$ for the inner automorphism on the group.}
Denote the {\em centralizer\/} of $A$ in $G$ by:
%\[
%G_A := \Ker\big( \Ad(A)\big). \]
%\marginpar{b:I don't think $\Ker$ is correct. You want the set of fixed points,
%not the subset which gets mapped to $0$. I've been using $\Fix$ for this notation.}
\[ G_A := \Fix\big(\Inn(A)\big) \subgroup G, \]
where $\Fix(S)$ is the set of fixed points of $S \subseteq G$.

Define the commutator map of $G$:
\begin{align*}
G \times G &\xrightarrow{~\kappa~} G \\
(A,B) &\longmapsto [A,B] := ABA^{-1}B^{-1}. \end{align*}

%\marginpar{We use capital letters to denote elements of a general group, small letters for elements in the Lie groups and Greek letters for elements $\pi$ and in $\Gamma$.}

Suppose further that $G$ is a % linear algebraic
Lie %
group  with Lie algebra $\g$,
and denote the adjoint representation of $G$ on $\g$ by $\Ad$.
Identify $\g$ with the Lie algebra of {\em right-invariant vector fields\/} on $G$;
then for any right-invariant vector field $X \in \g$,
the element $\Ad(a)(X)$ equals the image of $X$ under left-multiplication
by $a$.
Denote the {\em centralizer\/} of $a$ in $\g$ by:
%\[
%G_A := \Ker\big( \Ad(A)\big). \]
%\marginpar{b:I don't think $\Ker$ is correct. You want the set of fixed points,
%not the subset which gets mapped to $0$. I've been using $\Fix$ for this notation.}
\[ \g_a := \Fix\big(\Ad(a)\big) \subseteq \g. \]

%We often use small letter to denote elements in Lie groups.  So for example, $K_c$ will %denote the centralizer of $c$ in $K$.

%Let $a \in G$.
%We denote the adjoint
%\begin{align*}
%\Ad(a) : \g  \lto \g, \ \ \ \Ad(a)X = aXa\inv.
%\end{align*}
%%We sometime shorten $aXa\inv$ to $X^a$.
Denote the trace of a matrix $a$ by $\Tr(a)$
and the $\lambda$-eigenspace of a matrix $a$ by $\Eig_\lambda(a)$,
for a scalar $\lambda\in\C$.

The notation $K$ and $\fk$ are reserved for compact Lie groups and its Lie algebra, respectively.
%If $a,b$ are elements of a group,
%define their {\em commutator\/} as $[a,b] := aba\inv b\inv$.

%For a smooth map $M\xrightarrow{~f~} N$ between smooth manifolds,
%denote its {\em differential\/} at $x\in M$ by:
%\[
%T_x M \xrightarrow{~d~} T_{f(x)} N \]
%where $T_xM$ denotes the {\em tangent space\/} of $x$ at $M$, etc.

If $(M,\Omega)$ is a symplectic manifold,  and %$f : M \lto \R$
$M \xrightarrow{~f~} \R$ is a smooth function,
denote its {\em Hamiltonian vector field\/} by $\Ham(f)$.
We denote the tangent space to a smooth manifold $M$ at a point $p\in M$
by $T_p M$.

When we say a set is {\em closed}, we mean it to be closed in the %usual
classical topology.
%\marginpar{b: changed ``usual'' to ``classical''}

%\section{The commutator map}

%$$
%\kappa : G \times G \lto G: \ \ \kappa(A,B) = [A,B] := ABA^{-1}B^{-1}.
%$$

%If $G$ is a Lie group and let $\Ad$ to be the adjoint map:
%$$
%\Ad : G \times \g \lto \g.
%$$
%We shorten the notation $\Ad(g, X)$ to $X^g$ when there is no confusion.
%%Let
%$$P_i : G \times G \lto G, \ \ \  1 \le i \le 2$$
%be the projection to the $i$-th factor.

%For the rest of the paper, we will assume $g=1=n$ and $\SU(3) = K \subset G = \SL(3,\C)$.

%\begin{prop}[Goto~\cite{Go49}]\label{prop:surjective}
%For any compact semi-simple Lie group $K$,
%the commutator map $K \times K \xrightarrow{~\kappa~} K$ is surjective.
%\end{prop}
%\begin{proof}
%See \cite{Go49}.
%\end{proof}
\section{Character varieties and the mapping class group}\label{sec:char}
We fix a basepoint on the boundary of $\surf : = \surf_{1,1}$.
%\marginpar{We fix the basepoint at the boundary because we don't want the $\Gamma$-action to move it.}
The fundamental group $\pi := \pi_1(\surf)$ is isomorphic to the rank 2 free group $\Ftwo$, generated
%by two simple closed curves $\alpha$ and $\beta$ intersecting at one point.
by homotopy classes of
oriented based loops $\alpha$ and $\beta$. %and they (could) intersect at the basepoint.
We often do not distinguish elements in $\pi$ from corresponding oriented based loops on $\surf$.

We write
$$
\pi = \la \alpha, \beta, \sigma | \kappa(\alpha, \beta) = \sigma \ra,
$$
where $\sigma$ is the boundary element.
%\marginpar{I am changing back to $\sigma$ because $\gamma$ is used later for elements in $\Gamma$.}
In this way, we have
$$
R := \Hom(\pi, K) \cong K \times K \ \ \ \text{ and } \ \ \ \M := R/K,
$$
where the $K$-action is by conjugation.
%\marginpar{We will use $\CC$ for conjugacy classes since we have only one boundary component.}

In our case, we have only one boundary circle and we let
$\CC \subseteq K$ be a conjugacy class and $c \in \CC$.  Then the relative representation variety and character variety are
$$
R_c : = \Hom_\CC(\pi,K) := \kappa^{-1}(c) \ \text{ and } \ \M_c := R_c(K)/K_c.
$$
Again, the $K_c$-action is by conjugation.
In this way, a representation $\rho \in R_c$ corresponds to $(a,b) \in K \times K$ such that $\kappa(a,b) = c$.
Notice that $\M_c$ is usually and equivalently defined as
$$
\M_c = \kappa^{-1}(\CC)/K.
$$
The space $\M_c$ has a natural symplectic structure $\Omega$ \cite{Go84, Go97}.

The diffeomorphism group of $\Sigma$ (fixing the boundary, hence, also the basepoint) acts on $\pi$ and this action descends to a $\Gamma$-action on $\pi$, fixing the conjugacy class of $\sigma$.
%\marginpar{b:need to require that the diffeormorphism preserves the basepoint in order
%to get an action on $\pi$}
This further induces an action:
\begin{align*}
\M_c \times \Gamma &\longrightarrow \M_c \\
([\rho], \gamma) &\longmapsto [\rho \circ \gamma].
\end{align*}
%\marginpar{mu and omega are defined in the introduction section}
The $\Gamma$-action leaves $\Omega$ and $\mu$ invariant \cite{Go97}.
For any oriented simple closed curve $\alpha$ on $\surf$, denote by $\tau_\alpha$ the Dehn twist along $\alpha$.
The mapping class group $\Gamma$ %
contains all % is generated by
Dehn twists; %
indeed the Dehn twists {\em generate\/} $\Gamma$
(although we do not need this fact).
Denote by $\Scal$ the set of homotopy classes of oriented simple closed curves on $\Sigma$.
%\marginpar{Dehn twists have directions}

\section{Compact Lie groups}
%\marginpar{Again, I am satisfied with this section, but feel free to add more.}
This section reviews well known facts that are used in the proofs.
{\em Generic elements\/} are introduced; these are regular elements
which are dense in their maximal tori,
and provide nontrivial dynamics.

\subsection{Regularity}

Suppose $M$ is an irreducible algebraic set over $\R$ or $\C$ and $M^s \subset M$ its singular locus.
Then $U  = M \setminus M^s$ is a smooth manifold and Zariski dense in $M$.
The smooth structure on $U$ gives rise to the Lebesgue measure class on $U$ and on $M$,
by assigning $M^s$ to be a null set.
We shall always mean this class, which coincides with the measure class discussed in the introduction \cite{Hu95}.
%We shall always mean this measure class.
%\marginpar{Change $X$ to $M$ to  be consistant.}

Let $G$ be a linear %connected
semisimple algebraic group over $\C$ of rank $r$
and $K < G$ a maximal compact subgroup.
The corresponding Lie algebras are denoted by $\fk$ and $\g$, respectively.

Recall that an element $a \in K$ is {\bf regular} if $K_a$ has dimension $r$.
In general, $\dim(K_a) \ge r$, and $K_a$  contains a maximal torus of $K$.
%%%% changed future tense to present 03262020
%Hence $K_a$ is a maximal torus (that is, a Cartan subgroup)  of $K$.
An element $a \in K$ is regular if and only if $K_a$ is a maximal torus (that is, a Cartan subgroup) in $K$.

Recall that an action on a topological space is {\em minimal\/} if every orbit is dense.
If $a\in K$, denote the Haar measure on $K_a$ by
$\mu_{K_a}$ and the pushforward  $(L_b)_*\mu_{a}$ under left-multiplication
$K_a \xrightarrow{~L_b~} bK_a$ by $\mu_{ba}$.
%\marginpar{cumbersome notation}

\subsection{Genericity}
In general regularity is too weak a notion for dynamical complexity.
We introduce a notion of {\em genericity\/} which is more useful for
constructing nontrivial dynamics.

Let $(a,b)\in K\times K$.
Then the cyclic group $\la a\ra$ acts on the left-coset $b K_a$ by:
\[
b\zeta \xmapsto{~a^n~}  b\zeta a^n \]
where $n\in\Z$ and $\zeta\in K_a$.

\begin{prop}\label{prop:irrational rotation}
Let $a\in K$ be a regular element.
For any $b\in K$,
the following conditions are equivalent:
\begin{itemize}
\item the cyclic group $\la a \ra < K_a$ is Zariski dense in $K_a$;
\item the cyclic group $\la a \ra < K_a$ is dense in $K_a$;
\item the action of $\la a\ra$ on $bK_a$ is minimal;
\item the action of $\la a\ra$ on $(bK_a,\mu_{ba})$ is ergodic.
%\item $a$ has infinite order.
\end{itemize}
\end{prop}
%\marginpar{infinite order seems particular to regular elements in $\SUthree$}
In this case, we say that $a$ is {\em generic.\/}
The proof of Proposition~\ref{prop:irrational rotation} uses standard facts about compact abelian Lie groups,
such as:
\begin{lem}\label{lem:Density}
A cyclic subgroup of $K_a$ is dense in the classical topology if and only if it is dense in the Zariski topology.
\end{lem}
%\marginpar{I found an even simpler proof by moving a reference here.  Please check.}
\begin{proof}
Any set that is classically dense is also Zariski dense since the Zariski topology is coarser than the classical topology. We now show the converse.
Clearly the cyclic group $\la a\ra \subset K_a$, and its Zariski-closure is also an abelian subgroup.
Its closure in the classical topology
%\marginpar{``closure'' should default to the classical topology in this paper}
\[ H :=\overline{\la a\ra} \] is a closed abelian subgroup of $K_a$.
Now every compact linear group is Zariski-closed
(See Onishchik-Vinberg~\cite{OV90}, \S 4.4, Theorem~5, pp.133--134).  Hence $H$ is Zariski closed in $K_a$.  Since ${\la a\ra}$ is Zariski dense in $K_a$ and $H \supseteq {\la a \ra}$, $H=K_a$.
%Its identity component
%$H^0$ is a compact connected Lie subgroup of $K$
%which lies in a maximal torus of $K$.  By (2.19, \cite{Ad96}), $H^0$ is a compact torus, hence, is Zariski closed.
%Let $T$ be a maximal torus of $K$ containing $H^0$.
%Denote the Lie algebras of $H$ and $T$ by $\fh$ and $\ft$ respectively.
%The exponential map $\ft \xrightarrow{~\exp~} T$ is a covering space
%with kernel a lattice $\Lambda\subset\ft$.
%Then $H^0  = \exp(\fh)$ is compact if and only if $\fh \cap \Lambda$ is a lattice in $\ft$.
%Equivalently, $\Lambda$ contains a basis of $\fh$.

%The closed subgroup $H$ is an extension of the torus $H^0$ by a finite abelian group
%$H/H^0 \cong \pi_0(H)$.  Hence $H$ is the union of finite copies of $H^0$.  This implies that $H$ is Zariski closed.  Hence $H = T$.

%Since the quotient $T/H^0$ is also a torus,  the homomorphism
%\[
%H/H^0   \longrightarrow  T/H^0 \]
%is an embedding onto a finite subgroup of the quotient torus $T/H^0$.   Hence $H$ is Zariski closed and dense in $T$
%Thus the closure $\overline{\la a\ra}$ is determined by two finite subsets  $F_0,F_1\subset\Lambda$ and
%a positive integer $n\in \NN$ as follows:
%\begin{itemize}
%\item $H^0 := \exp( F_0\R) $;
%\item $H \  =\  (\frac1{n} F_1) \Z \ \cdot H^0$.
%\end{itemize}
%That is, $F_0$ generates the sublattice $\fh \cap \Lambda$ and the finite group
%$\pi_0\big(\overline{\la a\ra}\big)$ is generated by $\exp \frac1{n} F_1$.
%(Of course, $F_0,F_1,n$ are not unique.)
\end{proof}

\begin{proof}[Proof of Proposition~\ref{prop:irrational rotation}]
The proof now follows from Lemma~\ref{lem:Density} and
the fact that dense subgroups of the torus act minimally
(see Katok-Hasselblatt \cite [\S 1.4 (p.28) ]{KH95})
and ergodically (\cite[Proposition 4.2.2 (p.147)]{KH95},
or Walters \cite[Theorem 1.9 (p.30)]{Wa82}.
\end{proof}

\section{Infinitesimal transitivity and Hamiltonian flows}\label{sec:IT}
In this section, we let $G$ be a semi-simple complex algebraic Lie group.
%\marginpar{I moved up this because the items need to be introduced before we discuss twist flows.  Perhaps we should combine (short) section with the next.}
Let $M$ be a symplectic manifold and $M \xrightarrow{~f~} \R$ a smooth function.  
Denote by $\Ham(f)$ the associated Hamiltonian vector field.

\begin{defin}
Let $M$ be a manifold and $\Fm$ be a set of
real smooth $\R$-functions on $M$ such that at $x \in M$, the
differentials $df(x)$, for $f\in\Fm$, span the cotangent space
$T_x^*(M)$. Then $\Fm$ is said to be infinitesimally transitive at $x$.  $\Fm$ is infinitesimally transitive on $M$ if
$\Fm$ is infinitesimally transitive at all $x \in M$.
\end{defin}

\begin{prop}\label{prop:Ham}
Let $M$ be a connected symplectic manifold and $\Fm$ be infinitesimally transitive on $M$.
Then the group $\mH$ generated by the Hamiltonian flows
$\Ham(f)$ of the vector fields $\Ham(f)$, for $f\in\Fm$, acts
transitively on $M$.
\end{prop}
\begin{proof}
See Lemma 3.2 in \cite{GX11}.
\end{proof}

%\subsection{Invariant functions and centralizing one-parameter subgroups}\label{ssec:hamilton}
We now briefly review results of Goldman~\cite{Go86},
describing the flows generated by
the Hamiltonian vector fields by {\em simple\/}  closed curves on $\Sigma$.
In this case, the local flow of this vector field on $\M_c$
lifts to a flow on the representation variety $R_c$.
Furthermore this flow admits a simple description~\cite{Go86}
as follows.
\subsection{Invariant functions and flows on groups}
%Let  $\Ad$ be the adjoint representation
%of $G$ on its Lie algebra $\g$.
Let $G$ be a semi-simple complex Lie group with Lie algebra $\g$.
Then the adjoint representation $\Ad$ preserves a nondegenerate symmetric
bilinear form $\langle \_ , \_ \rangle$ on $\g$. In the case $G=\SL(n,\C)$, this is
\begin{equation*}
\langle X, Y\rangle := \Tr (X Y).
\end{equation*}

Let $G\xrightarrow{\f} \R$ be a function invariant under the inner automorphisms $\Inn(G)$.
Following \cite{Go86}, we describe how
$\f$ determines a way to associate to every element $x\in G$
a one-parameter subgroup
\begin{equation*}
\zeta^t(x) \, = \, \exp\big( t F(x) \big)
\end{equation*}
centralizing $x$.
Given  $\f$, define  its {\em variation function\/}
$ G \xrightarrow{F} \g$ by:
\begin{equation*}
\langle F(x), \upsilon \rangle = \frac{d}{dt}\bigg|_{t=0}
\f \big( x \exp(t\upsilon)\big)
\end{equation*}
for all $\upsilon\in\g$.
Invariance of $\f$ under $\Inn(G)$ implies that $F$ is $G$-equivariant:
\begin{equation*}
F( g x g^{-1}) = \Ad(g) F(x).
\end{equation*}
Taking $g = x$ implies that the one-parameter subgroup
\begin{equation}\label{eq:oneparameter}
\zeta^t(x) := \exp(t F(x))
\end{equation}
lies in the centralizer of $x\in G$.
%\marginpar{I thought about adding Sean's suggestion, but it seems that it doesn't make the calculation more clear unless we write down the Taylor series.}

Intrinsically, $F(x)\in\g$ is dual (by  $\langle \_, \_\rangle$) to the
element of $\g^*$ corresponding to
the left-invariant 1-form  on $G$ extending the covector
$df(x) \in T^*_x(G)$.

%\subsection{The coordinate ring}
\subsection{Invariant functions and centralizing one-parameter subgroups}
%\marginpar{The Hamiltonian function must be real.  There is an issue switching from $\SL(3,\C)$ case to $\SU(3)$ case.  I thought long and hard about that and believe that we need to take both the real and imaginary parts as functions.}

\label{ssec:hamilton}
%\marginpar{$\Scal$ was defined at the end of Section 2.  I added this ``Recall'' sentence to remind the reader what $\Scal$ stands for.}
Recall that $\Scal$ denotes the set of homotopy classes of oriented simple closed curves on $\Sigma$.
If $\alpha \in \Scal$ is an oriented homotopy class of based loops,
then $\f_\alpha$, the {\em trace function\/} of $\alpha$, is defined as:
\begin{align*}
\Hom(\pi,G) & \xrightarrow{~\f_\alpha~} \C \\
\rho & \longmapsto  \Tr\big(\rho(\alpha)\big).
\end{align*}
Since the function $G \xrightarrow{~\Tr~} \C$ is $\Inn(G)$-invariant,
$\f_\alpha$ defines a $\C$-valued function (also denoted by $\f_\alpha$) on $\M_c$.
Let 
$$\f_\alpha^R = \Re(\f_\alpha),  \ \ \ \f_\alpha^I = \Im(\f_\alpha).$$
Then $\f_\alpha^R$ and $\f_\alpha^I$ define $\R$-valued functions on $\M_c$.  

%Since the function $G \xrightarrow{~\Tr~} \C$ is $\Inn(G)$-invariant,
%$\f_\alpha$ defines a function (also denoted by $\f_\alpha$) on $\M_c$.
%This gives two $\R$-valued functions $\Re(\f_\alpha)$ and $\Im(\f_\alpha)$, namely the real and imaginary parts of $\f_\alpha$.  Let $\f$ be either one of the two.

%\subsection{Nonseparating loops}
%There are two cases, depending on whether
Let $\alpha \in \Scal$ and $\Sigma | \alpha$ denote the compact surface  %-with-boundary 
obtained by {\em splitting\/} $\Sigma$ along $\alpha$. 
The two components $\alpha_\pm$ of $\partial \Sigma | \alpha$ corresponding to $\alpha$
are the preimages of $\alpha\subset\surf$ under the quotient mapping $\Sigma|\alpha \longrightarrow\surf$.
%The boundary of $\Sigma | \alpha$ has two components, denoted by $\alpha_\pm$, corresponding
%The boundary of $\Sigma | \alpha$ has two components, denoted by $\alpha_\pm$, corresponding
%to $\alpha$. 
The original surface $\Sigma$ may be reconstructed as a quotient
space under the identification of $\alpha_-$ with $\alpha_+$.

The fundamental group $\pi$ can be
reconstructed from the fundamental group $\pi_1(\Sigma | \alpha)$ as
an HNN-extension:
\begin{equation}\label{eq:hnn}
\pi \;\cong\;   \bigg(\pi_1(\Sigma | \alpha) \amalg
\langle\beta\rangle \bigg)
\bigg/
\bigg(\beta \alpha_- \beta^{-1} = \alpha_+ \bigg).
\end{equation}
A representation $\rho$ of $\pi$ is determined by:
\begin{itemize}
\item the restriction $\rho'$ of $\rho$ to the subgroup
$\pi_1(\Sigma|\alpha)\subset\pi$, and
\item the value $\beta' = \rho(\beta)$
\end{itemize}
which satisfies:
\begin{equation}\label{eq:hnnrep}
\beta'  \rho'(\alpha_-) \beta'^{-1} = \rho'(\alpha_+).
\end{equation}
Furthermore any pair $(\rho',\beta')$ where $\rho'$ is a representation of
$\pi_1(\Sigma|\alpha)$ and $\beta'\in G$ satisfies \eqref{eq:hnnrep}
determines a representation $\rho$ of $\pi$.

Let $\f = \f_\alpha^R$ or $\f = \f_\alpha^I$.  
Define the {\em twist flow\/} $\xi_\alpha^t$, associated with $\f$ on $\Hom(\pi,\SU(3))$:
\begin{equation}\label{eq:hamtwist1}
\xi_\alpha^t(\rho):\gamma \longmapsto \begin{cases}  \rho(\gamma)
& \text{if}~ \gamma \in \pi_1(\Sigma|\alpha) \\
\rho(\beta) \zeta^t\big(\rho(\alpha_-)\big)
& \text{if}~ \gamma = \beta. \end{cases}
\end{equation}
where $\zeta^t$ is defined in \eqref{eq:oneparameter}.
This flow covers the flow generated
by $\Ham(\f)$ on $\M_c$
(See \cite{Go86}).

\subsection{Infinitesimal transitivity}
%\marginpar{I am not sure $\alpha \beta \alpha^{-1}$ corresponds to a simple closed curve.  The old generating set contains the commutator $[\alpha,\beta]$ instead.  The commutator does correspond to a simple closed curve, even though we don't need it.  Please check.}
Let $\XX$ be the Geometric Invariant Theory quotient of  $\Hom\big(\Ftwo,\SL(3,\C)\big)$ by 
$\Inn\big(\SL(3,\C)\big)$.
% and $\XX_\R :=\SUthree^{\times 2}/ \SUthree$.
%Denote by $\RR = \C[\M]$ and $\RR_c = \C[\M_c]$ the coordinate rings of $\M$ and $\M_c$, respectively.
%Let
%\[
%\Scal = \la \f_\alpha \in \RR: \alpha \in S \ra, \ \ \ \Scal_c = \la \f_\alpha \in \RR_c: \alpha \in S \ra.
%\]
Choose $\alpha, \beta \in \Scal$ as in Section \ref{ssec:hamilton} to correspond to curves with geometric
intersection number $1$ (or equivalently a free basis of $\Ftwo$).
%The structure of the $\SL(3,\C)$-character variety of the 1-holed torus is determined in \cite{La06,La07}.
Let
\[ \Gg := 
\{\alpha,\beta,\alpha \beta,\alpha\beta\inv,\alpha\beta\alpha\inv\beta\inv\}\subset \Scal \]
and 
$\Gg\inv := \{\gamma\inv : \gamma\in\Gg \}$. 
%\{\alpha,\beta,\alpha \beta,\alpha\beta\inv,\alpha\beta\alpha\inv,\beta\inv, \\
%& \qquad\qquad \alpha\inv,\beta\inv,\beta\inv\alpha\inv,\beta\alpha\inv,\beta\alpha\beta\inv\alpha\inv\} \subset 
%\Scal
%is symmetric, that is, invariant under $\gamma \longleftrightarrow\gamma\inv$.
Let  \[\Fm_{\alpha,\beta} := \{ \f_\gamma : \gamma\in\Gg\cup \Gg\inv\}. \]

\begin{thm}[Lawton~\cite{La06,La07}]\label{thm:top of SU(3) rep}
The coordinate ring $\C[\XX]$ is generated by $\Fm_{\alpha,\beta}$.
\end{thm}
Since $\Tr(A\inv) = \overline{\Tr(A)}$ for $A\in\SUthree$,
the set $\Fm_{\alpha,\beta}$ is invariant under complex conjugation.

%Since $\Gg$ is symmetric, 
%$\Scal_{\alpha,\beta}$ is invariant under complex conjugation, 

The relative $\SUthree$-character variety $\M_c$ of $\Ftwo$ embeds in $\XX$ as a real semialgebraic
subset. 
The regular functions on this $\R$-semialgebraic set are the real and imaginary parts 
of the restrictions to $\M_c$ of the regular functions on $\XX$.
%Due to the above symmetry, we can restrict to the real and imaginary parts of $f_
\begin{cor}\label{prop:IT g=1}
The set 
\[ 
\Fm_{\alpha,\beta}^\R := \{\f_\gamma^R, \f_\gamma^I : \gamma \in \Scal_{\alpha,\beta}\} \]
is infinitesimally transitive on $\M_c$.  
%In other words, the $\R$-valued coordinate ring $\R[\M_c]$ is generated by
%\begin{align*}
%\{\f^R, \f^I : \f \in \G\}.
%\mathcal{G} = \{\f_\alpha,\f_\beta,\f_{\alpha \beta}, \f_{\alpha^{-1}}, \f_{\beta^{-1}}, \f_{\beta \alpha^{-1}}, \f_{\alpha^{-1} \beta^{-1}}\}.
%\end{align*}
\end{cor}

\begin{proof}
Given the remarks preceding this corollary, the result follows from Theorem \ref{thm:top of SU(3) rep} and  \cite[Lemma 3.1]{GX11}.
\end{proof}

\section{The Dehn twists}\label{ssec:Dehn}
Let  $\alpha \in \Scal$ and $\tau_\alpha\in \Gamma$ be the corresponding Dehn twist.
%If $\alpha$ is essential,
%then $\tau_\alpha$ induces a nontrivial
%element of $\Out(\pi)$.
The fundamental group $\pi$ can be
reconstructed from the fundamental group $\pi_1(\Sigma | \alpha)$ as
an HNN-extension as in \eqref{eq:hnn}.
Then $\tau_\alpha$ induces the automorphism
$(\tau_\alpha)_*\in\Aut(\pi)$ defined by:

\begin{equation*}
(\tau_\alpha)_*:\gamma \longmapsto
\begin{cases}  \gamma
& \text{if}~ \gamma \in \pi_1(\Sigma|\alpha) \\
\gamma  \alpha
& \text{if}~ \gamma = \beta. \end{cases}
\end{equation*}
This further induces the map $(\tau_\alpha)^*$ on $\Hom(\pi,G)$ mapping
$\rho$ to:
\begin{equation}\label{eq:dehnnonsep}
(\rho)(\tau_\alpha)^*:\gamma \longmapsto
\begin{cases}  \rho(\gamma)
& \text{if}~ \gamma \in \pi_1(\Sigma|\alpha) \\
%%% \rho(\gamma)  \rho(\alpha)
\rho(\gamma)  \rho(\alpha)
& \text{if}~ \gamma = \beta. \end{cases}
\end{equation}
(See \cite{Go86}).

\subsection{Dehn twists and Hamiltonian twist flows}\label{subsec:DH}

Let $a := \rho(\alpha)$ and $b = \rho(\beta).$
Then
\begin{align*}
R_c & = \{\rho \in \Hom(\pi, K) : \kappa(\rho(\alpha), \rho(\beta)) = c\} \\ & = \{(a,b) \in K \times K : \kappa(a,b) = c\}.
\end{align*}
Let
%\marginpar{b:Maybe this is better notation.}
$$\H(a,b) := \{a\}\times bK_a, \ \ \ \ \H'(a,b) := aK_b \times \{b\}$$

\begin{prop}\label{prop:flow}
If $(a,b) \in R_c$, then
$\H(a,b), \H'(a,b) \subseteq R_c.$
\end{prop}
\begin{proof}
Suppose $t \in K_b$.  Then $at = ta$ and $t^{-1}a^{-1} = a^{-1}t^{-1}$.
Then
$$
k(a,bt) = a(bt)a^{-1}(bt)^{-1} = abta^{-1}t^{-1}b^{-1} = aba^{-1}b^{-1} = k(a,b) = c.
$$
Hence $(a,bt) \in R_c$.  The proof of $aK_b \times \{b\} \subseteq R_c$ is similar.
\end{proof}
%\marginpar{Appropriate changes with respect to the Hamiltonian functions again.}
\begin{prop}\label{prop: H-orbit}
If $(a,b) \in R_c$,
then the Hamiltonian flows of the vector fields $\Ham(\f_\alpha^R)$ and  $\Ham(\f_\alpha^I)$ preserve $[\H(a,b)]$.  Similarly, $\Ham(\f_\beta^R)$ and $\Ham(\f_\beta^I)$ preserve $[\H'(a,b)]$.
\end{prop}
\begin{proof}
This follows from \eqref{eq:hamtwist1} and by exchanging $\alpha$ and $\beta$.
\end{proof}

%\begin{prop}\label{prop: H-orbit}
%If $(a,b) \in R_c$,
%then the Hamiltonian flows of the vector fields $\Ham(\Re(\f_\alpha))$ and  $\Ham(\Im(\f_\alpha))$ preserve $[\H(a,b)]$.  Similarly, $\Ham(\Re(\f_\beta))$ and $\Ham(\Im(\f_\beta))$
%preserve $[\H'(a,b)]$.
%\end{prop}
%\begin{proof}
%This follows from \eqref{eq:hamtwist1} and by exchanging $\alpha$ and $\beta$.
%\end{proof}

\begin{cor}\label{cor:generic1}
If $a$ is generic, then $\la \tau_\alpha\ra$ acts ergodically on $\H(a,b)$.
If $b$ is generic, then $\la \tau_\beta\ra$ acts ergodically on $\H'(a,b)$.
\end{cor}
%\marginpar{b:The font seems to change here for capital ``O''.}

\begin{proof}
By \eqref{eq:dehnnonsep},
$\tau_\alpha(a,b) \in \H(a,b)$ and $\tau_\beta(a,b) \in \H'(a,b)$
The corollary then follows from Proposition \ref{prop:irrational rotation}.
\end{proof}

%Denote by $\Ham(\G)$ the set of Hamiltonian flows generated $\G$.

\section{The case of $K = \SU(3)$}
%\subsection{Relation to traces}
For the rest of this paper, we denote $\omega := e^{2\pi i/3}$ and the identity transformation by $\Id$.
In this section we fix $K = \SU(3)$.  

The classification of conjugacy classes of $K$ can be described in terms of the {\em trace\/} function
\[
K \xrightarrow{~\Tr~} \C. \]
Let $\Delta = \Tr(K)$.

If $\zeta_1,\zeta_2,\zeta_3\in\C$ are the eigenvalues of $a\in K$,
then they satisfy
\begin{equation}\label{eq:Eigenvalues}
\vert \zeta_1\vert = \vert \zeta_2\vert = \vert \zeta_3\vert = 1 \text{~and~} \zeta_1 \zeta_2 \zeta_3 = 1. \end{equation}
 The coefficients of the character polynomial $\chi_a$ are:
\begin{align*}
1 & = 1\\
\zeta_1 + \zeta_2 + \zeta_3 &  = \Tr(a) \\
\zeta_2\zeta_3 + \zeta_3\zeta_1 + \zeta_1\zeta_2 & = \overline{\Tr(a)} \\
\zeta_1 \zeta_2\zeta_3 & = 1.
\end{align*}
Therefore the  characteristic polynomial is:
\[
\chi_A(\lambda) \  = \
\lambda^3 \  - \ z\, \lambda^2 \ + \ \bar{z} \,\lambda \  - 1\]
where $z = \Tr(a)\in\C$.
Furthermore \eqref{eq:Eigenvalues} is equivalent to the condition:
\[
\vert z \vert^4 -  8\mathsf{Re} (z^3)  + 18 \vert z\vert^2  - 27 \le 0 \]
and this real polynomial condition exactly describes the image $\Delta \subseteq \C$.

%For rest of the paper, we specialize $K$ to $\SU(3)$.
%Suppose $A \in K$, then the characteristic polynomial of $A$ is
%\[
%f_a(x) =
%\Det( \lambda \I - a) =  1 +  \Tr(a) \lambda + \Tr(a^{-1}) \lambda^2 + \lambda^3.
%\]
%\marginpar{b:I think this was the negative since the so I changed the left-hand side}
%The conjugacy classes in $K$ are described by the function
%\begin{align*}
%%ch : K \lto \C^2, \ \ \ ch(A) = (f_1(A), f_2(A))$$
%K & \xrightarrow{~\chi~} \C^2 \\
%a &\longmapsto  \Big(\Tr(a), \Tr(a^{-1})\Big) \end{align*}
%which assigns to $a \in K$
%the non-trivial coefficients of its characteristic polynomial

\begin{figure}\label{fig:Delta}
%\centerline{\epsfig{figure=SUthreeTraces.pdf,height=6cm}}
 \includegraphics[scale=0.5]{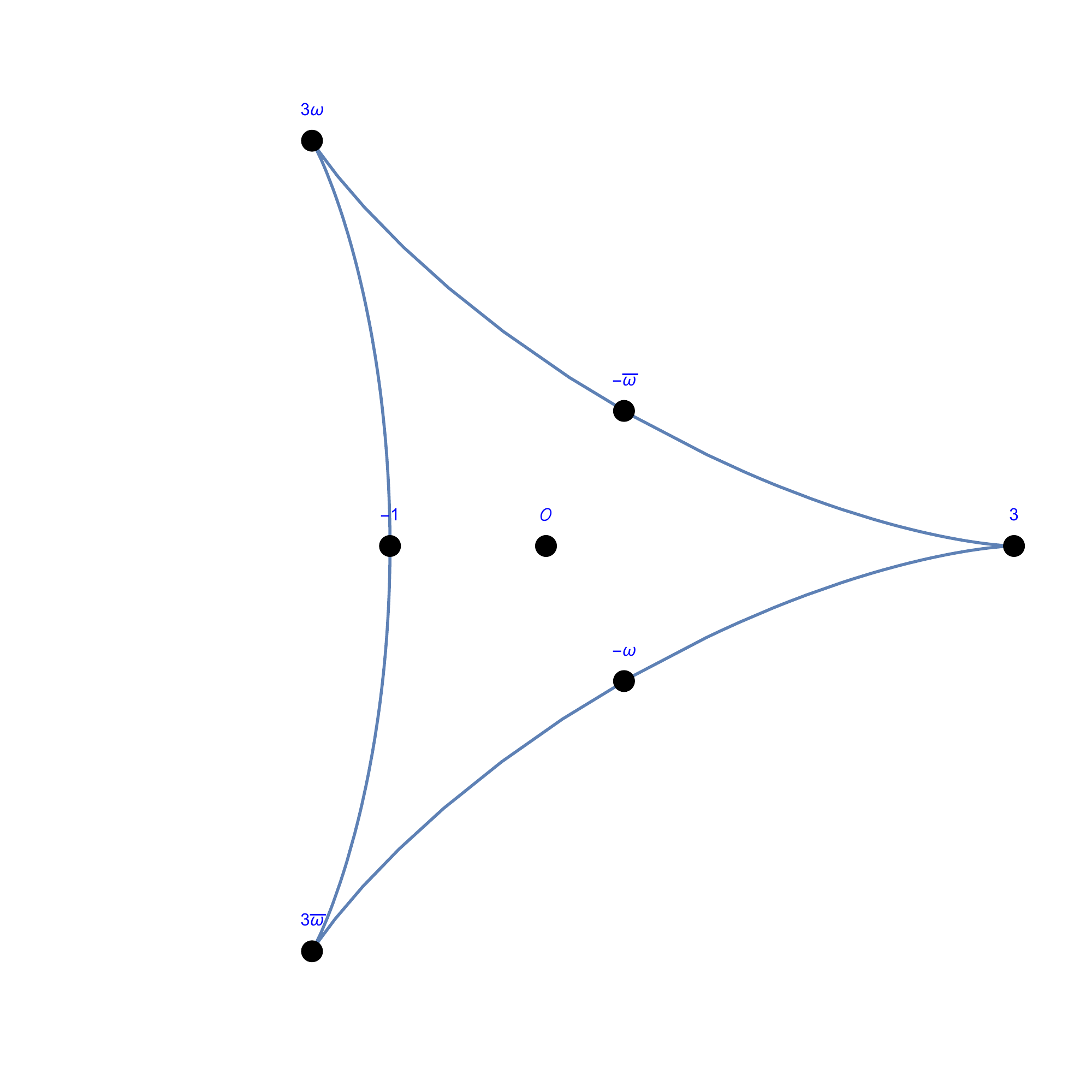}
 \caption{$\Delta$ the traces in $\SU(3)$.}
 \end{figure}

The traces of central elements are the vertices $3, 3\omega, 3\bar{\omega}$ of
$\Delta$.
The trace of a regular element of order three is the center $0$ of $\Delta$.
The traces of elements of order two are $-1, -\omega, -\bar{\omega}$,
the midpoints of the edges of $\partial \Delta$.

\begin{prop}\label{prop:Sard}
The map $\Tr$ is a local submersion at almost all points of $\Delta$.
\end{prop}
\begin{proof}
The map $\Tr$ is smooth.  Hence, by Sard's theorem, almost all points of $\C$ are regular values of $\Tr$ \cite[\S 1.7]{GP10}.  Hence $\Tr$ is a local submersion at almost all points of $\Delta$.
\end{proof}
%\marginpar{This is Bill's calculation, corrected.}
%\marginpar{I think we should prove the Remark if there is a short proof.}
\begin{rem}
It is not difficult to show that $\Tr$ has full rank in the interior of $\Delta$, but we only need Proposition \ref{prop:Sard}.  For a general discussion of $\Delta$ and Weyl chambers, see \cite{DK00}.
\end{rem}

\begin{prop} \label{cor:GenericityConullInDelta}
The image $\Tr(N)$ of the subset $N \subseteq K$ of generic elements
is conull in  $\Delta$.
\end{prop}
\begin{proof}
%Let
%$$
%\Delta' = \{\zeta \in \Delta : \Re(\zeta), \Im(\zeta), \frac{\Re(\zeta)}{\Im(\zeta)} \in \R \setminus \Q \}.
%$$
Let
$$
U = \{(\alpha, \beta) \in \R \times \R : \alpha, \beta, \frac{\alpha}{\beta} \in \R \setminus \Q\}
$$
Let $\Delta'\subset\Delta$ be the image of $U$ under the mapping
\begin{align*}
\R \times \R &\longrightarrow \C \\
(\alpha, \beta) &\longmapsto e^{2\pi i\alpha} + e^{2\pi i\beta} + e^{-2\pi i(\alpha +\beta)}
\end{align*}
Then $\Delta'$ is conull in $\Delta$ and $\Tr^{-1}(\Delta') \subseteq N$.
\end{proof}

%Explicitly, this right-invariant vector field $\xi$ equals:
%\[
%\xi(b) := (\mathsf{D}\mathfrak{R}_{a\inv b})_a(v) \in T_b G
%\]
%where $\mathfrak{R}_{a\inv b}$ denotes right-multiplication by $a\inv b$,
%and
%\[
%T_a G \xrightarrow{~(\mathsf{D}\mathfrak{R}_{a\inv b})_a ~} T_b G \]
% is its differential at $a$.
%Under the right-parallelism, right-multiplication preserves each element of $\g$, and
%left-multiplication by $a\in G$ acts by the adjoint representation $\mathsf{Ad}(a)$.

\section{Central fibers of $\kappa$}
Again in this section, we fix $K = \SU(3)$.
\subsection{The abelian representations}\label{sec:ab}
%\marginpar{This is the natural place to briefly introduce the abelian representations.}
The fiber $R_\I = \kappa\inv(\I)$ consists of commuting pairs $(a,b)$.
In this case, $a$ and $b$ lie in a maximal torus $\bT^2$.  Hence $\M_\I \cong (\bT^2 \times \bT^2)/W$, where $W$ is the Weyl group of $K$ acting diagonally (see \cite{FL14} for more discussion of these abelian character varieties)

%The other central fibers $\kappa\inv(\omega 1)$ and
%$\kappa\inv(\bar\omega 1)$ are more interesting:
%myab
\subsection{The non-abelian cases}\label{sec:nab}
\begin{prop}\label{prop:single point}
If $k = \omega\I$ where $\omega\neq 1$, then $\M_k$ consists of a single point.
Specifically, if $(a,b)\in R_k$, then there exists $g\in K$ such that
\[
a = g a_0 g\inv , \qquad b = g b_0 g\inv \]
where
\begin{equation}\label{eq:A0B0}
a_0 := \bmatrix 1 & 0 & 0 \\ 0 & \omega & 0 \\ 0 & 0 & \omega^2 \endbmatrix, \qquad
b_0 := \bmatrix 0 & 0 & 1 \\ 1 & 0 & 0 \\ 0 & 1 & 0 \endbmatrix. \end{equation}
%In particular, both $A$ and $B$ are each have order $3$ and generate a subgroup of order $27$.
\end{prop}
\begin{proof}
Suppose $(a,b)\in R_k$, that is,
\begin{equation} \label{eq:CommOmega}
a b a\inv b\inv = \omega\I.  \end{equation}
We first prove:
\begin{lem}\label{lem:OrderThree}
$a^3 = b^3 = \Id$ \end{lem}
\begin{proof}[Proof of Lemma~\ref{lem:OrderThree}]
By \eqref{eq:CommOmega} and taking traces,
\begin{equation*}
\Tr(a)  = \Tr(\omega b a b\inv) = \omega \Tr(a)    \ \text{ and } \
\Tr(a\inv)  = \Tr(\omega b\inv a\inv b) = \omega \Tr(a^{-1}).
\end{equation*}
Hence $\omega\neq 1$ implies that $\Tr(a) = \Tr(a\inv) = 0$.
Now apply the Cayley-Hamilton theorem:
\[ a^3 -  \Id = a^3 - \Tr(a) a^2 + \Tr(a\inv) a - \Det(a) \Id = 0. \]
The same argument applied to
$b   = \omega a\inv b a$ implies $b^3=\Id$, as claimed. \end{proof}
Returning to the proof of Proposition~\ref{prop:single point},
Lemma~\ref{lem:OrderThree} implies that
$a = \Inn(g) a_0$ for some $g\in K$, since $a\neq \I$.

We claim that $b = \Inn(g) b_0$.
%\marginpar{b:$b$ and $b_0$ are in the group, {\em not\/} the Lie algebra.
%$\Ad(g)$ acts on the Lie algebra.}
For convenience assume that $a = a_0$;
then we show that \eqref{eq:CommOmega} implies that
$b$ is the permutation matrix $b_0$ defined in \eqref{eq:A0B0}.

% \[ A = \bmatrix 1 & 0 & 0 \\ 0 & \omega & 0 \\0 & 0 & \omega^2 \endbmatrix \]
% since $A\neq \I$ and $\omega$ is a primitive cube root of $1$.

%\begin{lemma}\label{lem:PermuteEig}
Recall that $\Eig_\lambda(a)$ is the $\lambda$-eigenspace of $a$.
We claim that \eqref{eq:CommOmega} implies that
\begin{equation}\label{eq:PermuteEig}
b \big( \Eig_\lambda (a) \big) \ = \ \Eig_{\omega\lambda} (a). \end{equation}
To see this, rewrite \eqref{eq:CommOmega} as:
\begin{equation}\label{eq:ABBA}
a b \ = \ \omega ba  \end{equation}
Suppose that $v\in  \Eig_\lambda (a)$, that is,
\[
a v = \lambda v \]
so applying \eqref{eq:ABBA},
\[
a (b v) \ = \ \omega b a v \ = \ \omega\lambda (b v) \]
whence $b v \in \Eig_{\omega\lambda} (a)$, as claimed.

Since $a$ is the diagonal matrix $a_0$ defined by \eqref{eq:A0B0},
the lines
$\Eig_1(a), \Eig_\omega(a),$ and $\Eig_{\bar\omega}(a)$ are the three coordinate lines in $\C^3$.
Thus \eqref{eq:PermuteEig} implies that $b = b_0$,
concluding the proof of Proposition~\ref{prop:single point}.
\end{proof}

In particular, both $a$ and $b$ have order $3$.
Since $\kappa(a,b)$ has order 3 and is central, 
the subgroup $\langle a, b\rangle \subset K$ is a nonabelian group of order $27$,
a nontrivial central extension of $\Z/3 \oplus \Z/3$ by $\Z/3$.

\section{Ergodicity}
Let $K$ be any compact Lie group.
Each tangent space $T_a K$ identifies with the Lie algebra $\fk$ of {\em right-invariant vector fields:\/}
namely, a tangent vector $v\in T_a K$ identifies with the right-invariant vector $X \in \fk$ such that $X(a) = v$.  In this way, the differential of the commutator map $\kappa$ at
$(a,b) \in K \times K$ identifies with the linear map (see \cite{Go84, Go20, PX02a}):
\begin{align*}
%\T_{(a,b)}\big(G\times G\big) \leftrightarrow
 \D\kappa_{(a,b)} : \fk \oplus \fk & \longrightarrow \fk %\leftrightarrow \T_{[a,b]}
\\
 (X,Y) &\longmapsto \Ad(ba)\big((\Ad(b^{-1}) - \I\big)X  \\
 & \qquad \qquad + \big(\I - \Ad(a^{-1})\big)Y.
\end{align*}
From this formula, the following proposition holds:
%\marginpar{I moved up this cor to here because it is its natural place.}
\begin{prop}\label{prop:C-singular}
$\kappa$ is %smooth 
submersive %
at $(a,b)$ if and only if $\fk_a \cap \fk_b = 0$.
\end{prop}
%\begin{proof}
%The differential $$ of the commutator map $\kappa$ at
%$(a,b) \in K \times K$ identifies with :
%The proposition follows.
%\end{proof}
%\marginpar{Rewrote the standard representation for clarity.}
For the rest of this section, let $K = \SU(3)$ which comes with the standard representation $\Pi$ on $\C^3$.  An element in $(a,b) \in R_c$ corresponds to a representation $\rho$ of $\pi$ (see Section \ref{sec:char}).  Hence $\rho \circ \Pi$ is a representation of $\pi$.  Denote by $\M_c^i \subseteq \M_c$ the subspace of irreducible representation classes.

\subsection{Relation to the moduli space of $K$-bundles}  If we fix a complex structure on $\Sigma$, then we can construct the coarse moduli space $\N^{ss}$ of semi-stable parabolic $K$-bundles \cite{BR89} on $\Sigma$, endowed with a complex structure.  $\N^{ss}$ contains the subspace $\N^s$ of stable parabolic $K$-bundles.

%A representation $\rho : \pi \lto K$ is then an element in $R_c$ and

\begin{prop}\label{prop:connected}
The set of smooth points of $\M_c$ is a connected manifold.
\end{prop}
\begin{proof}
There is a homeomorphsim $\M_c \cong \N^{ss}$, restricting to a diffeomorphism $\M_c^i \cong \N^s$ (Theorem 1, \cite{BR89}).
The moduli space $\N^{ss}$ is irreducible and contains $\N^s$ as an open subvariety (Theorem II,  $\M_c$ \cite{BR89}).  Hence $\N^s$ is open and connected.  Hence $\M_c^i$ is open and connected. Proposition \ref{prop:C-singular} implies that
a point $[\rho] = [(a,b)] \in \M_c$ is a smooth point if and only if $\rho$ is irreducible, i.e. $\M_c^i$ is also the set of smooth points of $\M_c$.  The proposition follows.
\end{proof}

%\marginpar{What is $C_1$?}

For almost every conjugacy class $c$,
the action $\M_c \times \Gamma \to \M_c$ is ergodic \cite{PX02a}.
%\marginpar{b: Are we using right-actions?}
This section proves that this is true for {\em all} $c$.

Let $ c \in K$.  Up to conjugation,
\[
c = \bmatrix c_1 & 0 & 0 \\ 0 & c_2 & 0 \\ 0 & 0 & c_3  \endbmatrix.
\]
For $(a,b) \in R_c$, we have the natural map
$$
\iota : K_a \lto \H(a,b), \ \ \ \iota(t) = (a,bt).
$$

%Let $A \in \SU(3)$ and $ch(A) = (a_1, a_2)$ and $a_3 = \frac{1}{a_1+a_2}$.  Then $A$ is generic if and only if $\frac{a_i}{a_j} \not\in \Q$ for $i \neq j$.

%For the rest of this section, we assume $R_c$ to be the subset of its smooth points.
%\marginpar{I realized that we need to carefully distinguish $\T$ and $T$.}
Let $P_2 : K \times K \lto K$ be the projection to the second factor.
An element $\rho \in R_c$ corresponds to a pair $(a,b) \in K \times K$.  Let
$$\T = \Tr \circ P_2 : R_c \lto \Delta \ \ \text{ and } \ \  T := \iota \circ \T : K_a \lto \Delta.$$
%We say $T$ a submersion at $(a,b)$ if  $T$ is a local submersion at $(a,b)$.
%More precisely this means
%$
%\D T = \D \Tr \circ \D P_2
%$
%has full rank of 2 at $(a,b)$.
%The projection $P_2$ is always a submersion.  Hence
%$\D T$ has full rank if and only if $\D \Tr$ has full rank.

%n that case $(a,b)$ is in the interior of  %%
%the closed $2$-simplex $\Delta$.

%\begin{prop}\label{prop:single point}
%If $k = c_1 \I$ where $c_1 \neq 1$, then $\M_c$ consists of a single point.
%\end{prop}
%\begin{proof}
%Suppose $[\rho] \in \M_c$.  Then $\rho$ is irreducible.  Hence the conjugation action of $K$ on $R_c$ is free.  $R_c$ is connected \cite{PX03a}.  A dimension count shows that the $K$-action is transitive.
%\end{proof}

\begin{prop}\label{prop:submersion}
Let $c \not\in Z(K)$.
Then $(a,b) \in R_c$ exists such that:
\begin{enumerate}
\item $(a,b)$ is a smooth point;
\item $T$ is a submersion at $(a,b)$.
\end{enumerate}
\end{prop}

%If $c \not\in Z(K)$, then there is $(a,b) \in R_c$ such that $(a,b)$ is a smooth point and $T$ is a submersion at $(a,b)$.
%%then for each $\beta \in S$, there exists smooth $\rho \in R_c$ such that $\rho$ is irreducible and that $\rho(\beta)$ is generic.
%%Moreover $\rho$ is a smooth point in $R_c$.
%\end{prop}
\begin{proof}
Let $(a,b) \in R_c$ such that
\begin{equation}\label{eqn:ab}
a = \bmatrix 0 & 1 & 0 \\ 0 & 0 & 1 \\ 1 & 0 & 0  \endbmatrix, \ \ \
b = \bmatrix b_1 & 0 & 0 \\ 0 & b_2 & 0 \\ 0 & 0 & b_3  \endbmatrix.
\end{equation}
Then
\begin{equation}
c = \kappa(a,b) = \bmatrix \frac{b_2}{b_3} & 0 & 0 \\ 0 & \frac{b_3}{b_1} & 0 \\ 0 & 0 & \frac{b_1}{b_2}  \endbmatrix.
\end{equation}
This formula for $c$ implies that $\kappa$ is onto $K$.
\begin{rem}
A result of Goto \cite{Go49} states that for any compact semi-simple Lie group K, $K \times K \xrightarrow{~\kappa~} K$ is surjective.  Hence $R_c \neq \emptyset$ for all $c \in K$.
\end{rem}

%\marginpar{b:The phrase "By a direct calculation" is vague.}
%The element $a$ is regular and therefore lies in a unique maximal torus,
%which equals the subset of K$ consisting of matrices of the form
%\begin{equation}\label{eq:CentralizerOfa}
%\bmatrix u & v & w \\ w & u & v \\  v & w & u \endbmatrix
%\end{equation}

%By a direct calculation, $Z(\la a,b\ra) = Z(K)$.
%\marginpar{b:What if $b_1=b_2=b_3$?}
Note that $a$ is regular.  For (1), there are three cases for $b \in Z(K)$, $b_1 = b_2 \neq b_3$ (and its permutation variation) and $b$ being regular.

%We begin by computing the Jordan canonical form of $a$.  Let $a = p q p^{-1}$ where $q$ is diagonal.

If $b \in Z(K)$, then $\kappa(a,b) = \I = c$ and this violates our hypothesis of $c \neq Z(K)$.

If $b$ is regular, then $t \in \fk_b$ implies
\begin{equation}\label{eq:t}
t = \bmatrix t_1 & 0 & 0 \\ 0 & t_2 & 0 \\ 0 & 0 & t_3  \endbmatrix,  \ \ \ t_1 + t_2 + t_3 = 0.
\end{equation}
Then
$$
(\Ad(a) - \I)t = \bmatrix t_2 - t_1 & 0 & 0 \\ 0 & t_3 - t_2 & 0 \\ 0 & 0 & t_1 - t_3  \endbmatrix.
$$
Hence if $t \in \fk_a$, then $(\Ad(a) - \I)t = 0$.  This implies $\fk_a \cap \fk_b = 0$.
Hence, by Proposition \ref{prop:C-singular}, we conclude that $\kappa$ is regular at $(a,b)$.

If $b_1 = b_2 \neq b_3$ and $t \in \fk_b$, then
$$
t = \bmatrix t_{11} & t_{12} & 0 \\ t_{21} & t_{22} & 0 \\ 0 & 0 & t_{33}  \endbmatrix, \ \ \ t_1 + t_2 + t_3 = 0.
$$
Then
$$
(\Ad(a) - \I)t = \bmatrix t_{22} - t_{11} & -t_{12} & t_{21} \\ -t_{21} & t_{33} - t_{22} & 0 \\ t_{12} & 0 & t_{11} - t_{33}  \endbmatrix.
$$
Hence if $t \in \fk_a$, then $(\Ad(a) - \I)t = 0$ This implies $t = 0$.  Hence $\fk_a \cap \fk_b = 0$.
By Proposition \ref{prop:C-singular}, $\kappa$ is regular at $(a,b)$.  We conclude that in all cases $(a,b) \in R_c$ is a smooth point.

%Consider the following equation in $b_i$'s:
%\begin{equation}\label{eq:comm}
%\bmatrix c_1 & 0 & 0 \\ 0 & c_2 & 0 \\ 0 & 0 & c_3  \endbmatrix = c = [a,b] = \bmatrix \frac{b_2}{b_1} & 0 & 0 \\ 0 & \frac{b_3}{b_2} & 0 \\ 0 & 0 & \frac{b_1}{b_3}  \endbmatrix.
%\end{equation}
Notice that $\H(a,b) \subseteq R_c$.
%\marginpar{b:Now I"m confused about the $\H$ notation.
%Here $\alpha$ seems to have some meaning, before it was just part of the $O$-notation.}
We consider $T$ restricted to $\H(a,b)$.
%$$
%t = \bmatrix t_1 & 0 & 0 \\ 0 & t_2 & 0 \\ 0 & 0 & t_3  \endbmatrix.
%$$
Let
$$
p = \bmatrix 1 & \omega^2 & \omega \\ 1 & \omega & \omega^2 \\ 1 & 1 & 1  \endbmatrix \ \ \text{ and } \ \
q = \bmatrix 1 & 0 & 0 \\ 0 & \omega^2 & 0 \\ 0 & 0 & \omega  \endbmatrix.
$$
Then
$a = p q p^{-1}$.  The element $t \in K_a$ has the form $t = p t_b p^{-1}$, where $t_b \in K_b$ is diagonal.  %Let $z := \Tr(t_b) = \Tr(t)$.
Then
%From the fact that $\Det(b) = \Det(t) =1$, we determine the characteristic polynomial of $bt$ to be:
%\begin{equation}
%f_{bt}(x)  = 1 - \frac{1}{3} \Tr(a^{-1}) \Tr(t^{-1})x  + \frac{1}{3}\Tr(b)\Tr(t) x^2 - x^3,
%\end{equation}
%whence $\chi : P_2(\H(a,b)) \lto \Delta:$
$$
T(t) = \Tr(bt) = \frac{1}{3} \Tr(b) \Tr(t).
$$
By Proposition \ref{prop:Sard}, $T$ is a local submersion for all almost all $t \in \Delta$ unless $\Tr(b)=0$.  However $\Tr(b) = 0$ implies $c \in Z(K)$, a contradiction.  Hence $\Tr(b) \neq 0$ and $T$ is a local submersion for almost all $t$.
%Now we compute the differential $\chi_*$ with respect to $t$ and look for the locus where it does not have full rank.
%
%To do this, we differentiate $\psi_1$ and $\psi_2$.  Then
%$\chi_*$ is singular at $t$ if and only if $d\psi_1 \wedge d\psi_2 = 0$.  A direct calculation shows that
%$$
%d\psi_1 \wedge d\psi_2 = \frac{1}{9}\Tr(b) \Tr(b^{-1}) \sum_{i < j} (t_j - t_i) d t_i \wedge d t_j.
%$$
%Since $\Tr(b) = 0$ implies $b \in Z(K)$ and $b \neq \I$, it follows that $\Tr(b) = 0$ if and only if $\Tr(b^{-1}) = 0$.
%This implies that $\chi_*$ is onto unless $\Tr(b) = 0$ or $t = \I$.

\end{proof}
\begin{cor}
$\T$ is a local submersion for almost all points in $\H(a,b)$.
\end{cor}
\begin{proof}
Since $\D T = \D \T \circ \D \iota,$  $\D T_{t}$ being surjective implies $\D \T_{(a,bt)}$ is surjective.
\end{proof}
%Since $P_2$ is a submersion, it follows that $T|_{\H(a,b)}$ is a local submersion at $(a,bt)$ unless $b \in Z(K)$.  Since $\H(a,b) \subseteq R_c$, the proposition follows.

\begin{cor}\label{cor:generic2}
There is a conull set $V \subset R_c$ such that $b$ is generic for almost all $(a,b) \in V$.
\end{cor}
\begin{proof}
The subset containing points at which a map is locally submersive is Zariski open.
Hence,
by Proposition~\ref{prop:submersion}, there exists smooth and Zariski open $V \subseteq R_c$ such that $T|_V$ is a submersion.  Let $Q \subseteq R_c$ be the smooth part.  By Proposition~\ref{prop:connected}, $Q$ is connected, hence, irreducible.  Since Zariski open subset of a smooth irreducible variety is conull in the Lebesgue class, $V$ is conull.  The map $T|_V$ is a fibration over an open domain of $\Delta$.
The corollary follows from Proposition~\ref{cor:GenericityConullInDelta}.
\end{proof}

%\begin{cor}\label{cor:GH}
%Suppose $c \not\in Z(K)$.  Suppose $\beta \in S$ and $\phi : \M_c \lto \R$ is a $\mu$-measurable function.  If $\phi$ is $\tau_\beta$-invariant, then $\phi$ is $\Ham(\f_\beta$)-invariant.
%\end{cor}
%\begin{proof}
%By Proposition~\ref{prop:submersion}, $R_c$ contains a smooth point $(a,b)$ with $b$ generic.
%By Proposition~\ref{prop:connected} and Corollary~\ref{cor:generic2},
%$b \in K$ is generic for almost all $(a,b)\in R_c$.  Hence $b \in K$ is generic for almost all $[(a,b)] \in \M_c$.  By Proposition~\ref{prop:irrational rotation}, \ref{prop:flow} and Corollary~\ref{cor:generic1}, $\tau_\beta$-orbit is dense in $\H'(a,b)$ and $\phi$ is $\Ham(\f_\beta$)-invariant.
%\end{proof}

%\begin{cor}\label{cor:generic2}
%There is a conull set $V \subset R_c$ such that $b$ is generic for almost all $(a,b) \in V$.
%\end{cor}
%\begin{proof}
%The subset containing points at which a map is locally submersive is Zariski open.
%Hence,
%by Proposition~\ref{prop:submersion}, there exists smooth and Zariski open $V \subseteq R_c$ such that $T|_V$ is a submersion.  Let $Q \subseteq R_c$ be the smooth part.  By Proposition~\ref{prop:connected}, $Q$ is connected, hence, irreducible.  Since Zariski open subset of a smooth irreducible variety is conull in the Lebesgue class, $V$ is conull.  The map $T|_V$ is a fibration over an open domain of $\Delta$.
%The corollary follows from Proposition~\ref{cor:GenericityConullInDelta}.
%\end{proof}
%
\begin{cor}\label{cor:GH}
Suppose $c \not\in Z(K)$.  Suppose $\beta \in \Scal$ and $\phi : \M_c \lto \R$ is a $\mu$-measurable function.  If $\phi$ is $\tau_\beta$-invariant, then $\phi$ 
is $\Ham(\f_\beta^R)$-invariant and $\Ham(\f_\beta^I)$-invariant.
\end{cor}

\begin{proof}
By Proposition~\ref{prop:submersion}, $R_c$ contains a smooth point $(a,b)$ with $b$ generic.
By Proposition~\ref{prop:connected} and Corollary~\ref{cor:generic2},
$b \in K$ is generic for almost all $(a,b)\in R_c$.  Hence $b \in K$ is generic for almost all $[(a,b)] \in \M_c$.  By Proposition~\ref{prop:irrational rotation}, \ref{prop:flow} and Corollary~\ref{cor:generic1}, $\tau_\beta$-orbit is dense in $\H'(a,b)$ and $\phi$ is $\Ham(\f_\beta^R)$-invariant and $\Ham(\f_\beta^I)$-invariant.
\end{proof}

With the notation we have adopted, we restate Theorem \ref{thm:main} as:
\begin{thm}
The action $\M_c \times \Gamma \to \M_c$ is ergodic.
\end{thm}

\begin{proof}
%\marginpar{I just learned the proof for the case of $c=\I$.  Perhaps you can improve on this.}
%\marginpar{$\Gamma$-acts on the right, so it corresponds to the dual linear action of $\SL(2,\Z)$  It is not the most satisfying exposition, but any improvement I tried lengthened it significantly.}
Suppose that $c = \I$.
%Furthermore $\bT^2 \times \bT^2 = \R^2/\Z^2 \times \R^2/\Z^2$.
Identifying $\R^2$ with its dual $(\R^2)^*$, the group $\SL(2,\Z)$ has the standard dual linear action on $\R^2$ which induces the diagonal action on $\bT^2 \times \bT^2$.  This $\SL(2,\Z)$-action is known to be ergodic because $\SL(2,\Z)$ contains hyperbolic elements, meaning that the eigenvalues of these elements do not have absolute value $1$ \cite[\S 4]{BS15}.

There is an isomorphism \cite{FM12}
$
\iota : \Gamma \stackrel{\cong}{\lto} \SL(2,\Z).
$
By Section \ref{sec:ab}, $\M_c \cong (\bT^2 \times \bT^2)/W$.  The $\Gamma$-action on $\M_c$ lifts to an action on $\bT^2 \times \bT^2$.
Moreover, this $\Gamma$-action is equivariant with respect to $\iota$.  Hence the $\Gamma$-action on $\M_c$ is ergodic.

%Choose $\gamma \in \Gamma$ such that $h = \iota(\gamma)$ is a hyperbolic element in $\SL(2,\Z)$ (this means that the eigenvalues of $h$ do not have absolute value 1).  Then the cyclic subgroup $\la h \ra$ acts ergodically on $\bT^2 \times \bT^2$.  Hence the $\Gamma$-action on $\M_c$ is ergodic.

Suppose that $c \in Z(K)$ and $c \neq \I$, then $\M_c$ is a single point by Proposition~\ref{prop:single point} and the statement is trivially true.

Suppose $c \not\in Z(K)$.
Recall that $\mH$ is the group generated by all Hamiltonian flows $\Ham(\f_\beta^R)$ and $\Ham(\f_\beta^I)$ where $\beta \in \Scal$.  Let $\phi : \M_c \lto \R$ be a $\Gamma$-invariant $\mu$-measurable function.  By Corollary \ref{cor:GH}, $\phi$ is $\mH$-invariant. By Corollary~\ref{prop:IT g=1} and Proposition~\ref{prop:Ham}, $\phi$ is constant almost everywhere.  Our theorem follows.
\end{proof}

\bibliographystyle{alpha}
\bibliography{bib}

\end{document}